\documentclass[10pt]{amsart}
\usepackage{xypic,somedefs,tracefnt,latexsym,rawfonts,graphicx,amsfonts,amssymb,latexsym,enumerate,amscd}
\makeatletter
\@namedef{subjclassname@2020}{%
  \textup{2020} Mathematics Subject Classification}
\makeatother
\def\cal{\mathcal}
\def\Bbb{\mathbb}

\def\r{\rangle}
\def\l{\langle}

\def\wt{\widetilde}
\def\p{\partial}

\newtheorem{thm}{Theorem}[section]
\newtheorem{prop}[thm]{Proposition}

\newtheorem{lemma}[thm]{Lemma}

\newtheorem{cor}[thm]{Corollary}
\newtheorem{defn}[thm]{Definition}
\newtheorem{rem}[thm]{Remark}

\numberwithin{equation}{section}
\makeatletter
\newcommand{\colim@}[2]{%
  \vtop{\m@th\ialign{##\cr
    \hfil$#1\operator@font colim$\hfil\cr
    \noalign{\nointerlineskip\kern1.5\ex@}#2\cr
    \noalign{\nointerlineskip\kern-\ex@}\cr}}%
}
\newcommand{\colim}{%
  \mathop{\mathpalette\colim@{\rightarrowfill@\textstyle}}\nmlimits@
}
\makeatother
\begin{document}
\date{\today}
\title[Orbifold braid groups]{Orbifold braid groups}
\author[S.K. Roushon]{S.K. Roushon}
\address{School of Mathematics\\
Tata Institute\\
Homi Bhabha Road\\
Mumbai 400005, India}
\email{roushon@math.tifr.res.in} 
\urladdr{http://www.math.tifr.res.in/\~\!\!\! roushon/}
\begin{abstract}
  The orbifold braid groups of two dimensional orbifolds  
  were defined in \cite{All02} to understand
  certain Artin groups as subgroups of some suitable orbifold braid groups.
  We studied orbifold braid groups in some more detail in \cite{Rou20} and
  \cite{Rou21}, to prove the Farrell-Jones Isomorphism conjecture for
  orbifold braid groups and as a consequence for some Artin groups.
  In this article we apply the results from \cite{Rou20} and \cite{Rou21},
  to study two aspects of the orbifold braid groups. First we show
  that the homomorphisms induced on the orbifold braid
  groups by the inclusion maps of a generic class of
  sub-orbifolds of an orbifold are injective.
 Then, we prove that the centers of most of the orbifold braid groups are trivial.
\end{abstract}

\keywords{Braid groups, configuration spaces, orbifold braid groups.}

\subjclass[2020]{Primary: Secondary:}
\maketitle


\section{Introduction}
Let $M$ be a connected two dimensional orbifold (see \cite{Sco83}). Consider the following
 hyperplane complement in $M^n$.

$$PB_n(M)=\{(x_1,x_2,\ldots, x_n)\in M^n\ |\ x_i\neq x_j\ \text{for}\ i\neq j\}.$$
\noindent
$PB_n(M)$ is an orbifold  and called the {\it configuration space} of ordered $n$-tuples 
of pairwise distinct points of $M$. There is an obvious action of the
symmetric group $S_n$ on $n$ letters, on
$PB_n(M)$ by permuting the coordinates. The quotient orbifold $PB_n(M)/S_n$ is denoted by
$B_n(M)$. The quotient map $PB_n(M)\to B_n(M)$ is an orbifold covering map with
$S_n$ as the group of orbifold covering transformations. This gives the following
exact sequence of orbifold fundamental groups.

\begin{align}\label{1.1}
  \xymatrix{1\ar[r]&\pi_1^{orb}(PB_n(M))\ar[r]&\pi_1^{orb}(B_n(M))\ar[r]&S_n\ar[r]&1.}\end{align}

\begin{defn}(\cite{All02}){\rm The orbifold fundamental group 
    $\pi_1^{orb}(PB_n(M))$, denoted by ${\cal {PB}}_n(M)$, is called the
    {\it pure orbifold braid group} of $M$ on $n$ strings, and
the orbifold fundamental group $\pi_1^{orb}(B_n(M))$, denoted by ${\cal B}_n(M)$, is called the
{\it orbifold braid group} of $M$ on $n$ strings.}\end{defn}

In the particular case $M={\Bbb R}^2$, classically, the corresponding
groups are called the {\it pure braid groups} and
the {\it braid groups}, respectively. The 
(pure) braid groups were first defined by 
Artin in \cite{Ar47}, and subsequently generalized to the case of
$2$-manifolds by Fox and Neuwirth in \cite{FoN62}. In \cite{All02}, Allcock
made a further generalization and considered two dimensional
orbifolds, as described above, to give a braid type pictorial representation of certain Artin
groups. See \cite{Tit66}, \cite{BS72} and \cite{Bri71} for some background on Artin groups.

The study of the (pure) braid groups is an important subject
in mathematics. An enormous amount of work has been done during the
last several decades, exposing its connections with several areas of
mathematics and physics. For some of the early fundamental works, other than the
ones mentioned above, see
\cite{FN62}, \cite{FB62}, \cite{Bir69}, \cite{Bir73}, \cite{Ga69} and \cite{Gol73}.

Let ${\cal C}'_0$ be the class of all genus zero connected $2$-manifolds with at least one
puncture or one boundary component, and ${\cal C}'_1$ be the class 
of all connected $2$-manifolds of 
genus greater or equal to one. In both the cases if there are infinitely
many punctures, then we assume that they are obtained after removing a
discrete subset. Since we need that in a compact part of the manifold there
should be only finitely many punctures.

Let ${\cal C}_0$ and ${\cal C}_1$ be the
classes of all two dimensional orbifolds with
only cone type singularities, and whose underlying spaces 
belong to ${\cal C}'_0$ and ${\cal C}'_1$, respectively.

The celebrated fibration theorem of Fadell and Neuwirth for manifolds (see
\cite{FN62}) gave an important method  
for better understanding of the configuration spaces of a manifold.
As a consequence of the fibration theorem, for each $S\in {\cal C}'_0\cup {\cal C}'_1$ the
following exact sequence can be deduced.

\begin{align}\label{1.2}
  \xymatrix{1\ar[r]&\pi_1(PB_{n-r}(\wt S))\ar[r]&\pi_1(PB_n(S))\ar[r]&\pi_1(PB_r(S))\ar[r]&1.}\end{align}

\noindent
Above, the homomorphism
$\pi_1(PB_n(S))\to \pi_1(PB_r(S))$ is induced by
the restriction to $PB_n(S)$, of the projection $S^n\to S^r$
to the first $r$ coordinates. And $\wt S=S-\{r-\text{points}\}$, that
is, the fiber over a base point in $PB_r(S)$. Recall that, by the Fadell-Neuwirth fibration
theorem (see \cite{FN62}) the map $PB_n(S)\to PB_r(S)$ is a fibration and an important consequence,
obtained from the long exact sequence of homotopy
groups of this fibration, is that $\pi_k(PB_n(S))=1$, for
all $k\geq 2$. This is the reason the
above long exact sequence took the short form, as in (\ref{1.2}).
But in the case of $M\in {\cal C}_0\cup {\cal C}_1$,
the corresponding statement that $\pi^{orb}_k(PB_n(M))=1$, for
all $k\geq 2$, is not yet known. See the Asphericity conjecture in \cite{Rou21}.  We gave some
examples of orbifolds where this is true (see Proposition 2.4 in
\cite{Rou21}).
Furthermore, in the orbifold category there does not exist any
suitable notion of a 
fibration which would make $PB_n(M)\to PB_r(M)$ a fibration.
Rather, it is expected that $PB_n(M)\to PB_r(M)$ is a kind of quasifibration (see
\cite{Rou20} and the Quasifibration conjecture in \cite{Rou21}). Nevertheless,
in \cite{Rou20} and \cite{Rou21} we proved the existence of the following 
exact sequence of pure orbifold braid groups, analogous to (\ref{1.2}),  
for $M\in {\cal C}_0$ and $M\in {\cal C}_1$, respectively.

\begin{align}\label{1.3}
  \xymatrix{1\ar[r]&{\cal {PB}}_{n-r}(\wt M)\ar[r]&{\cal {PB}}_n(M)\ar[r]&{\cal {PB}}_r(M)\ar[r]&1.}\end{align}
                                                       
\noindent
Here, $\wt M$ is obtained from $M$ by removing $r$ regular points, that is, it is the
fiber over a regular (base) point of $PB_r(M)$.

Together with the work of Allcock (\cite{All02}),
the above exact sequences were required to prove the
Farrell-Jones Isomorphism conjecture for all (pure) orbifold braid
groups, and as a consequence for a class of Artin groups of
finite, complex and affine types (see \cite{Rou20} and
\cite{Rou21}).

For another application of (\ref{1.3}), see \cite{Rou22} or Remark \ref{last}.

In this article, using the exact sequences (\ref{1.1}) and
(\ref{1.3}), we will study some more properties of the (pure)
orbifold braid groups. More specifically, we will prove that,
for  $M\in {\cal C}_0\cup {\cal C}_1$,  
the homomorphisms induced by the inclusion maps of connected sub-orbifolds $N$ of $M$ on their (pure)
orbifold braid groups, are injective whenever $\pi_1^{orb}(N)\to
\pi_1^{orb}(M)$ is injective (Theorem \ref{thm1}).
Furthermore, we will show that the centers of ${\cal {PB}}_n(M)$
and ${\cal {B}}_n(M)$ are 
trivial except for few cases (Theorem \ref{center}). These results and their proofs are motivated by
some analogous results of Paris and Rolfsen (\cite{PR99}) for 
$2$-manifolds. 

\section{Main results}

We start with the following definition of a sub-orbifold of a two
dimensional orbifold. First, recall that in dimension $2$, the
underlying space of an orbifold has a manifold structure (see \cite{Sco83}).

\begin{defn}{\rm 
Let $M\in {\cal C}_0\cup {\cal C}_1$. A {\it sub-orbifold} $N$ of $M$
is a
two dimensional sub-manifold of
the underlying space of $M$ satisfying the following conditions.

$\bullet$ $\overline N-N$ is a $1$-dimensional manifold, that is, $M-N$ does
not contain any isolated point. We call the components of
$\overline N-N$, the {\it boundary components} of $N$ ( or of $\overline N$).

$\bullet$ No boundary component of $\overline N$ contains
any of the cone points of $M$.

$\bullet$ A boundary component of $\overline N$ either lies in the interior of $M$
or is a boundary component of $M$.

$\bullet$ $N$ has the induced orbifold structure from $M$.}\end{defn}

\subsection{Injectivity}
We first identify the sub-orbifolds $N$ of $M$, so that the inclusion $N\subset M$ 
will induce injective homomorphisms on their (pure) orbifold braid groups.
A necessary condition is that on the orbifold fundamental group 
level the inclusion $N\subset M$ induces an injective
homomorphism, since ${\cal {PB}}_1(X)=\pi_1^{orb}(X)$ for any orbifold $X$.
We will see that this condition is also sufficient.

\begin{defn}\label{good} {\rm Let $M\in {\cal C}_0\cup {\cal C}_1$. A sub-orbifold $N$ of
    $M$ is called {\it nice} if
    the components of $\overline {M-N}$ satisfy the following conditions.

    $\bullet$ If a component does not contain any cone point, then it is not
    simply connected.

    $\bullet$ If the underlying space of a component is simply connected, then it 
    contains at least two cone points.}\end{defn}

\begin{rem}\label{injective}{\rm In Proposition \ref{sub}
    we will see that if $\pi_1^{orb}(N)$ is infinite, 
then we need $N$ to be nice, to make
sure that $\pi_1^{orb}(N)\to \pi_1^{orb}(M)$ is injective.
Note that, in Definition \ref{good}, if $\pi^{orb}_1(N)$ is finite, 
then $N$ is either a smooth disc or a disc with a single cone
point in its interior. There is nothing to prove when $N$ is a
smooth disc. In the second case we will prove in
Proposition \ref{sub} that, $\pi^{orb}_1(N)\to \pi_1^{orb}(M)$ is
again injective.}\end{rem}

We make one more relevant remark which we need later on.

\begin{rem}\label{nice-remark}{\rm Let $N$ be a sub-orbifold of $M$
    and $p\in N$ be an interior point. Then,  
    clearly $\wt N:=N-\{p\}$ is a sub-orbifold of $\wt M:=M-\{p\}$.
    Furthermore, since $\overline {\wt M-\wt N}=\overline {M-N}$, if
    $N$ is nice in $M$, then $\wt N$ is also nice in $\wt M$.}\end{rem}

Let $n\leq m$ and $N\subset M$ be a 
connected sub-orbifold. Choose $m$ regular points $s_1,s_2,\ldots,
s_m\in M$ in the interior of $M$, 
such that $s_1,s_2,\ldots, s_n$  lie in the interior of $N$, and $s_{n+1}, s_{n+2},\ldots, s_m$ lie
in the interior of $M-N$.
We consider the orbifold fundamental groups of $PB_m(M)$ ($B_m(M)$) and $PB_n(N)$ ($B_n(N)$) with
respect to the base points $(s_1,s_2,\ldots, s_m)$ and $(s_1,s_2,\ldots, s_n)$, respectively.

\begin{thm}\label{thm1} Let $M\in {\cal C}_0\cup {\cal C}_1$ and
  $N$ be a connected nice sub-orbifold of $M$. Then, the
  following inclusion induced homomorphisms are injective.

  \begin{align}\label{PB}
    \xymatrix{{\cal {PB}}_n(N)\to {\cal {PB}}_m(M).}\end{align}
  
\begin{align}\label{B}
\xymatrix{{\cal B}_n(N)\to {\cal B}_m(M).}\end{align}
\end{thm}

Paris and Rolfsen proved the above theorem in \cite{PR99} for $2$-manifolds 
$M$ and sub-manifolds $N$, assuming that no component of $\overline {M-N}$ is a disc.
See \cite{PR99} for more details.

\subsection{Center} We start with the following definition.

\begin{defn}{\rm A {\it simple surface} $S$ is a connected $2$-manifold whose
fundamental group has a non-trivial center. In
other words, $S$ belongs to the following manifolds or its interiors.

$${\Bbb {RP}}^2, C:={\Bbb S}^1\times I, T:={\Bbb S}^1\times {\Bbb S}^1,$$
$$Mb:={\Bbb S}^1\hat{\times} I (\text{M\"{o}bius band}), K:={\Bbb S}^1\hat{\times} {\Bbb S}^1 (\text{Klein bottle}).$$}\end{defn}

Note that, ${\Bbb {RP}}^2$ does not belong to ${\cal C}_0\cup {\cal C}_1$.

Paris and Rolfsen (\cite{PR99}) proved that the center of the (pure) surface braid groups of
compact large surfaces (see Definition in $\S$ 1.3 in \cite{PR99}) are trivial. We
prove an analogous result for (pure) orbifold braid groups, in the following theorem.

\begin{thm}\label{center} Let $M\in {\cal C}_0\cup {\cal C}_1$ and assume the following.

  $\bullet$ If $M$ has no cone point, then it is not a simple surface.

  $\bullet$ If the underlying space of $M$ is simple, 
  then it has at least one cone point.

  $\bullet$ If the underlying space of $M$ is simply connected (note that
  $M\neq {\Bbb S}^2$), then it has at least two cone points.

  Then, the
  center of ${\cal {PB}}_n(M)$ is trivial. Furthermore, assuming that, either $n=1$ or 
$n\geq 3$, the center of ${\cal B}_n(M)$ is trivial.\end{thm}

Using the fact that ${\cal B}_n(M)$ is torsion free, for compact 
    large surfaces $M$, it was shown in \cite{PR99} that the center of ${\cal B}_n(M)$ is
    trivial. Since the orbifold braid groups are in general not 
    torsion free, we will use  
    the fact that the symmetric group $S_n$ has trivial center for
    $n\geq 3$, to draw this conclusion. We do not know much about 
    the center of ${\cal B}_2(M)$ for $M\in {\cal C}_0\cup {\cal C}_1$.
    Also Theorem \ref{center} is not true when $M$ is a disc or a disc with one cone point
    of order $q\geq 2$. The center of ${\cal B}_n({\Bbb D}^2)$,
    for $n\geq 2$, is known to be infinite cyclic (see \cite{Cho48}).
    For the disc $M$ with one cone point,  
    ${\cal {PB}}_1(M)={\cal B}_1(M)={\Bbb Z}_q$. Computations of the center of
    ${\cal B}_n(M)$ for the cylinder and for the torus are
    explicitly done in \cite{PR99}.
    
\section{Some basic results}
In this section we prove two propositions which are required to prove Theorems
\ref{thm1} and \ref{center}. This will help us to start a method of induction
 for the proofs.

\subsection{Injectivity for $n=1$}
It is well known that if $N$ is a connected non-simply connected sub-manifold of a
connected $2$-manifold $M$, then the inclusion induced homomorphism
$\pi_1(N)\to \pi_1(M)$ is injective if and only if no component of $M-N$ is
simply connected. The main idea behind the proof is that a boundary component  
of $N$ is $\pi_1$-injective. Therefore, by the Van-Kampen theorem a component of $M-N$ is
simply connected if and only if $\pi_1(N)\to \pi_1(M)$ is not injective.
The same idea can be used to prove the
following proposition.

\begin{prop}\label{sub} Let $M\in {\cal C}_0\cup {\cal C}_1$ and
  $N\subset M$ be
  a nice connected sub-orbifold. Then, the inclusion induced homomorphism
  $\pi_1^{orb}(N)\to \pi_1^{orb}(M)$ is injective.\end{prop}

For the proof of the proposition we will need the following lemma.

\begin{lemma}\label{lemmasub} Let $M\in {\cal C}_0\cup {\cal C}_1$. If the underlying
  space of $M$ is simply connected (that is, if it is a disc), then assume that it has at least
  two cone points. Let $\p$ be a circle boundary component of $M$. Then, 
  $\pi_1(\p)\to \pi_1^{orb}(M)$ is injective, that is $\p$ is
  $\pi_1^{orb}$-injective.\end{lemma}

\begin{proof} First, suppose that the underlying space of $M$ is a disc and 
  there are $k$ cone points of orders $q_1,q_2,\ldots, q_k$,
  ($k\geq 2$), in its interior. Then,
  
 $$\pi_1^{orb}(M)\simeq \l \alpha_1\ |\ \alpha_1^{q_1}=1\r* \l
 \alpha_2\ |\ \alpha_2^{q_2}=1\r*\cdots  *\l \alpha_k\ |\
 \alpha_k^{q_k}=1\r.$$

\noindent
Here, for $i=1,2,\ldots, k$, $\alpha_i$ is the loop around
the $i$-th cone point. See Figure 1. Then, the boundary $\p$ represents
the element $\alpha_1\alpha_2\cdots\alpha_k$, which is obviously of
infinite order, since $k\geq 2$. Therefore, $\p$ is $\pi_1^{orb}$-injective.

 \medskip
\centerline{\includegraphics[height=4cm,width=4cm,keepaspectratio]{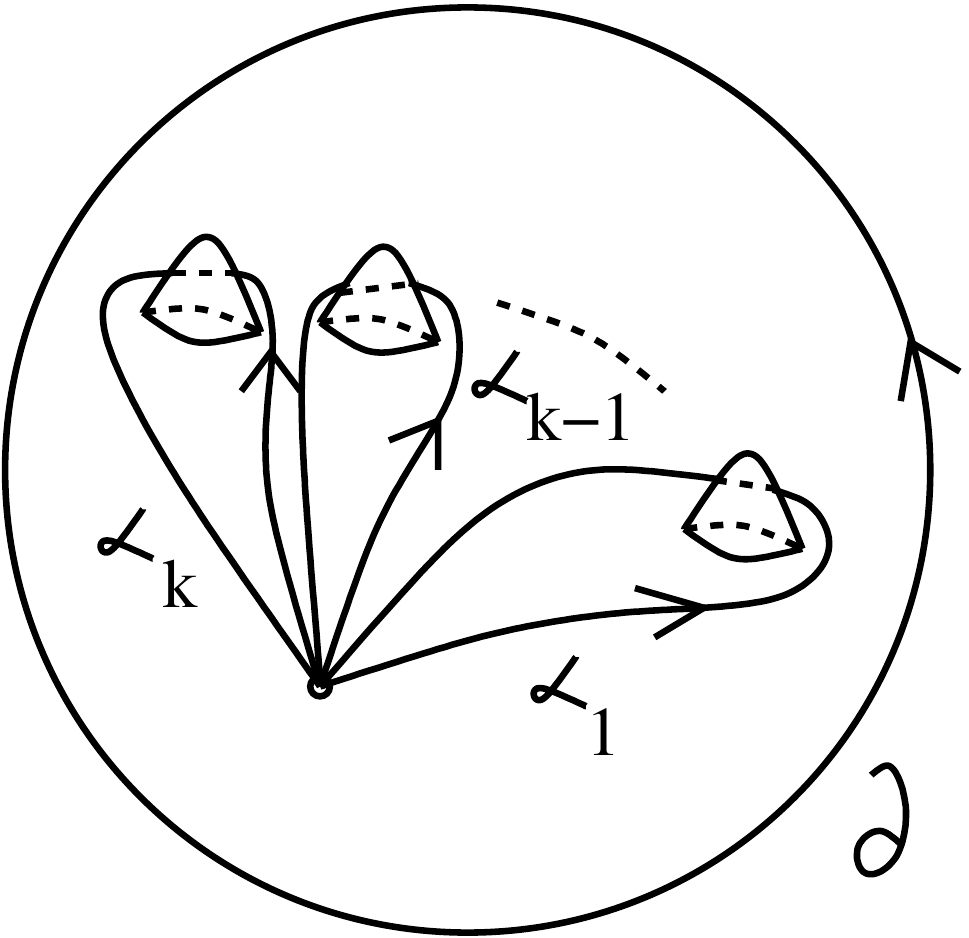}}

\centerline{Figure 1: Disc with $k$ cone points.}

\medskip

Next, assume that the underlying space $\check{M}$ of $M$ is either
of genus zero and non-simply connected or of
 genus $\geq 1$. Then, it is well known that $\pi_1(\p)\to
 \pi_1(\check{M})$ is injective.
 
 \begin{align}\label{p-injective}
   \xymatrix{&\pi_1^{orb}(M)\ar[d]&\pi_1(\breve{M})\ar[l]\ar[dl]\\
   \pi_1(\p)\ar[ur]\ar[r]&\pi_1(\check{M}).&}\end{align}

 Now, we have the diagram (\ref{p-injective}) with commutative
 triangles. Here $\breve{M}$ is equal to $M$ minus the cone
 points. Note that, $\pi_1^{orb}(M)$ is obtained from
 $\pi_1(\breve {M})$, by adding the relations $\alpha^q=1$,
 for each loop $\alpha$ around the puncture of $\breve {M}$, which
 was obtained by removing a cone point of order $q$ (see \cite{Thu91}).
 This describes the top horizontal homomorphism. Clearly, the right hand
 side slanted homomorphism is obtained by adding the
 relation $\alpha=1$ to $\pi_1(\breve {M})$, where
 $\alpha$ is as above. Furthermore, the
 vertical (surjective) homomorphism is obtained by equating the
 elements
 of $\pi_1^{orb}(M)$, which are loops around
 the cone points, to the trivial element. 

Since, the bottom horizontal homomorphism is injective, so is the left
hand side slanted one.

This completes the proof of the lemma.
 \end{proof}

 Now, we come to the proof of Proposition \ref{sub}.
 
 \begin{proof} [Proof of Proposition \ref{sub}]  First,
  we assume that $\pi_1^{orb}(N)$ is infinite. Then, it satisfies the
  hypothesis of Lemma \ref{lemmasub}. Hence, the boundary
components of $N$ are $\pi_1^{orb}$-injective. Furthermore,
again by Lemma \ref{lemmasub}, the boundary of a disc
  with  cone points in its interior is $\pi_1^{orb}$-injective if and
  only if there are at least two cone points in the disc.  The
  rest of the argument follows from the Van-Kampen theorem for
  orbifolds.

  Next, assume that $\pi_1^{orb}(N)$ is finite. Then, clearly $N$
  is either a smooth disc or a disc with one cone point of order $q$ (say).
  If $N$ is a smooth disc, then there is nothing to prove. So
  assume the later and let $\p$ be the boundary circle of $N$.
  Then, by Lemma \ref{lemmasub}, $\p$ is $\pi_1^{orb}$-injective in $\overline {M-N}$.
  Furthermore, $\pi_1^{orb}(M)$ is obtained from
  $\pi_1^{orb}(\overline {M-N})$, by attaching the relation
  $\alpha^q=1$, where $\alpha$ represents $\p$. Since  $\p$ is
  $\pi_1^{orb}$-injective, $\alpha$ has order $q$ in
  $\pi_1^{orb}(M)$. On the other hand
  $\pi_1^{orb}(N)=\l\alpha\ |\ \alpha^q=1\r$. It is now clear that 
$\pi_1^{orb}(N)\to \pi_1^{orb}(M)$ is injective.
\end{proof}

\subsection{Center for $n=1$}
First we state a couple of easy to prove lemmas on the center of an amalgamated free
  product, and the center of an extension of groups. Then, we recall a well-known theorem on
  the center of one-relator groups. Let $Z(-)$ denotes the center of a group. 

  \begin{lemma} \label{amal} If $G$ is an amalgamated free product of
    a non-trivial group with trivial center and an another group, then the center
    of $G$ is trivial.\end{lemma}

\begin{proof} Let $G=G_1*_HG_2$, then using the normal form of elements of
  a generalized free product, it can be deduced that
  $Z(G)=Z(G_1)\cap Z(G_2)\cap H$. Therefore, if one of $G_1$ and $G_2$
  is non-trivial and has trivial center, then $G$ also has trivial center.\end{proof}

\begin{lemma}\label{extension} Consider an exact sequence of groups.

  \begin{align}\label{center-extension}
    \xymatrix{1\ar[r] & K\ar[r] & G\ar[r]^p&H\ar[r]&1.}\end{align}

  If $Z(K)=Z(H)=\l 1\r$, then $Z(G)=\l 1\r$.\end{lemma}

\begin{proof} Let $g\in Z(G)$, then $p(g)\in Z(H)$ since $p$ is
  surjective. Since $Z(H)=\l 1\r$, $p(g)=1$, and hence, $g\in
  K$. Consequently, $g\in Z(K)$ and hence, $g=1$, since $Z(K)=\l
  1\r$.\end{proof}

\begin{thm}\label{one-relator} (\cite{KMS60}, \cite{N73}) Let $G$ be a non-cyclic one-relator
  group with a non-trivial element of finite order. Then, the center of $G$ is trivial.\end{thm}

In the following proposition we describe exactly when the center of
the orbifold fundamental group of an orbifold in 
${\cal C}_0\cup {\cal C}_1$ is trivial.

\begin{prop}\label{center-trivial} Let $M\in {\cal C}_0\cup {\cal C}_1$ and assume the
  following.

  $\bullet$ If $M$ has no cone point, then it is not a simple surface.

  $\bullet$ If the underlying space of $M$ is a simple surface, 
  then it has at least one cone point.
  
$\bullet$ If the underlying space of $M$ is simply connected (note that
  $M\neq {\Bbb S}^2$), then it has at least two cone points.
  
  Then, 
  $\pi_1^{orb}(M)$ has trivial center.\end{prop}

\begin{proof}
  If $M$ is smooth and not a simple surface, then it is well-known
  that $\pi_1^{orb}(M)$ has trivial center. Since
  the only $2$-manifold $M\in {\cal C}_0\cup {\cal C}_1$ with nontrivial center belongs to the following
  $2$-manifolds and their interiors: the cylinder ($C$), the M\"{o}bius band ($Mb$), the
  torus ($T$) and the Klein bottle ($K$).

  The proof in Case 2 below also
  gives a direct proof of this fact, in the particular case when $M$ is smooth and
  not a simple surface.

We are now ready to complete the proof of the proposition in the case when
$M$ has at least one cone point. We divide the proof in the
  following two cases.

  \noindent
  {\bf Case 1.} Assume that the underlying space of $M$ is a simple surface.
  Hence, it is one
  of the manifolds $C,T,Mb\ \text{or}\ K$ or its interiors.

  First, assume that each of them has at least
  two cone points. Take a disc $D\subset M$ in the interior of
  $M$ which contains the
  cone points in its interior. Then, split $M$ into two pieces along
  the
  boundary of $D$. Consequently, by Lemma
  \ref{lemmasub} and by the
  Van-Kampen theorem, $\pi_1^{orb}(M)$ satisfies
  the hypothesis of Lemma \ref{amal}, and hence has a trivial center.

         \medskip
\centerline{\includegraphics[height=8cm,width=12cm,keepaspectratio]{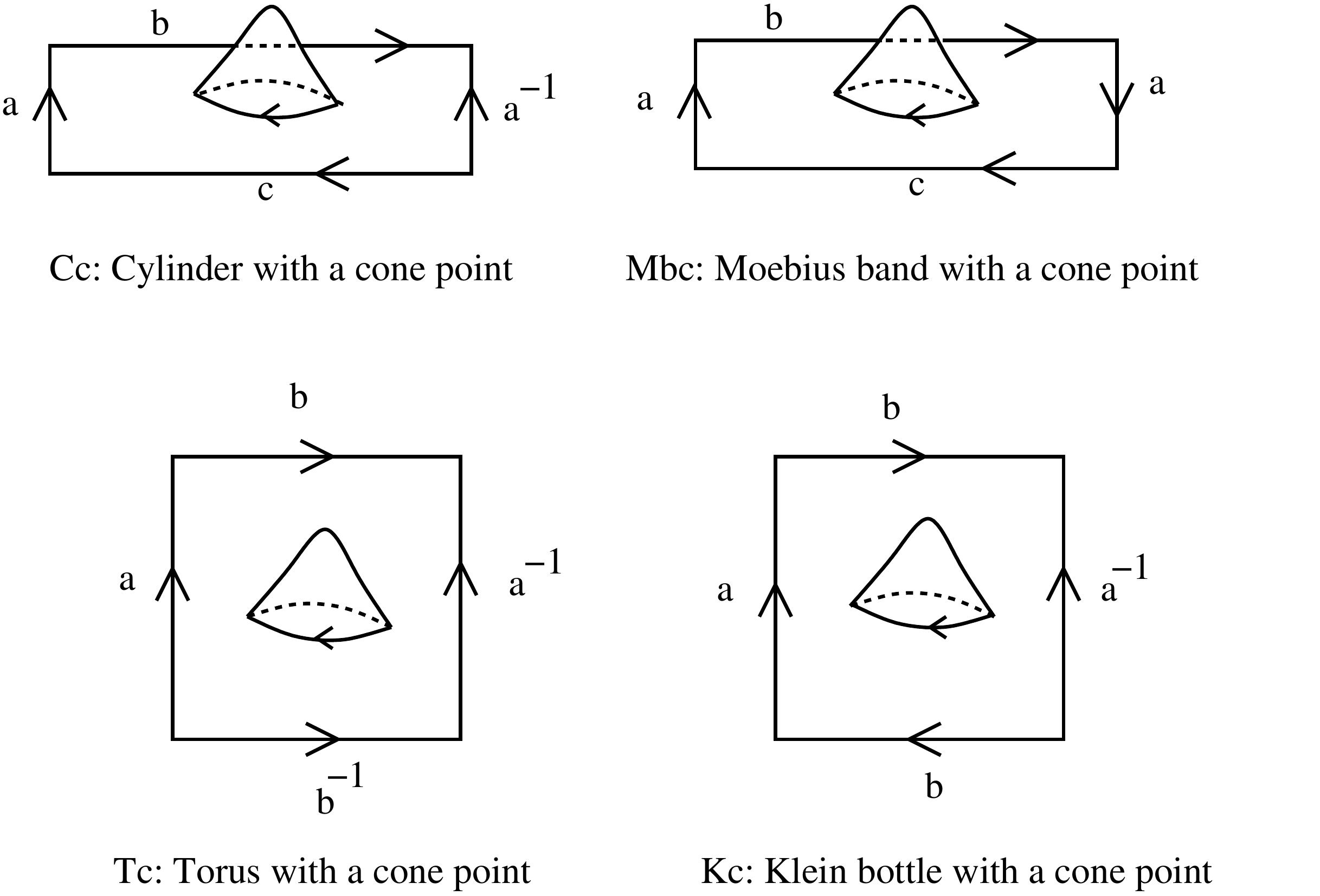}}

\centerline{Figure 2: Orbifolds with one cone point and simple underlying space.}

\medskip
  
  Next, assume that $M$ has only one cone point of order $q\geq 2$.
  We denote these four compact orbifolds by $Cc$, $Tc$, $Mbc$ and $Kc$. Since, the
  orbifold fundamental groups of the non-compact
  cylinder and M\"{o}bius band fall in the corresponding compact cases, we need
  to consider only the above four compact cases.
  We write down their orbifold fundamental groups explicitly as follows.

  $$\pi_1^{orb}(Cc)=\{a,b,c\ |\ (aba^{-1}c)^q=1\}, \pi_1^{orb}(Mbc)=\{a,b,c\ |\ (abac)^q=1\},$$
$$\pi_1^{orb}(Tc)=\{a,b\ |\ (aba^{-1}b^{-1})^q=1\}, \pi_1^{orb}(Kc)=\{a,b\ |\ (aba^{-1}b)^q=1\}.$$

The above calculation can be done using the pictorial presentation in
Figure 2 of
the orbifolds and the Van-Kampen theorem.
 
  It is now easy to
  see that $\pi_1^{orb}(M)$ is a non-cyclic group, by computing the
  abelianization of $\pi_1^{orb}(M)$ from the above presentations.

  Therefore, in all the four cases above the corresponding groups are
  non-cyclic, one-relator and each has an element of finite order.
  Hence, by Theorem \ref{one-relator}, the
  center of $\pi_1^{orb}(M)$ is trivial.

  {\bf Case 2.} Assume that the underlying space of $M$ is not a simple surface. Therefore,
  by hypothesis the underlying space of $M$ is one
  of the following types.

  $\bullet$ It is simply connected (but not ${\Bbb S}^2$) and it has
  at least two cone points.

  $\bullet$ It has genus zero, has at least two punctures (or boundary
  components) 
  in the non-orientable case, and
  has at least three punctures (or boundary components) in the orientable case.
  
  $\bullet$ It has genus
  $\geq 2$ in the orientable case.

  $\bullet$ It has genus $\geq 3$ in the non-orientable
  case.

              \medskip
\centerline{\includegraphics[height=12cm,width=12cm,keepaspectratio]{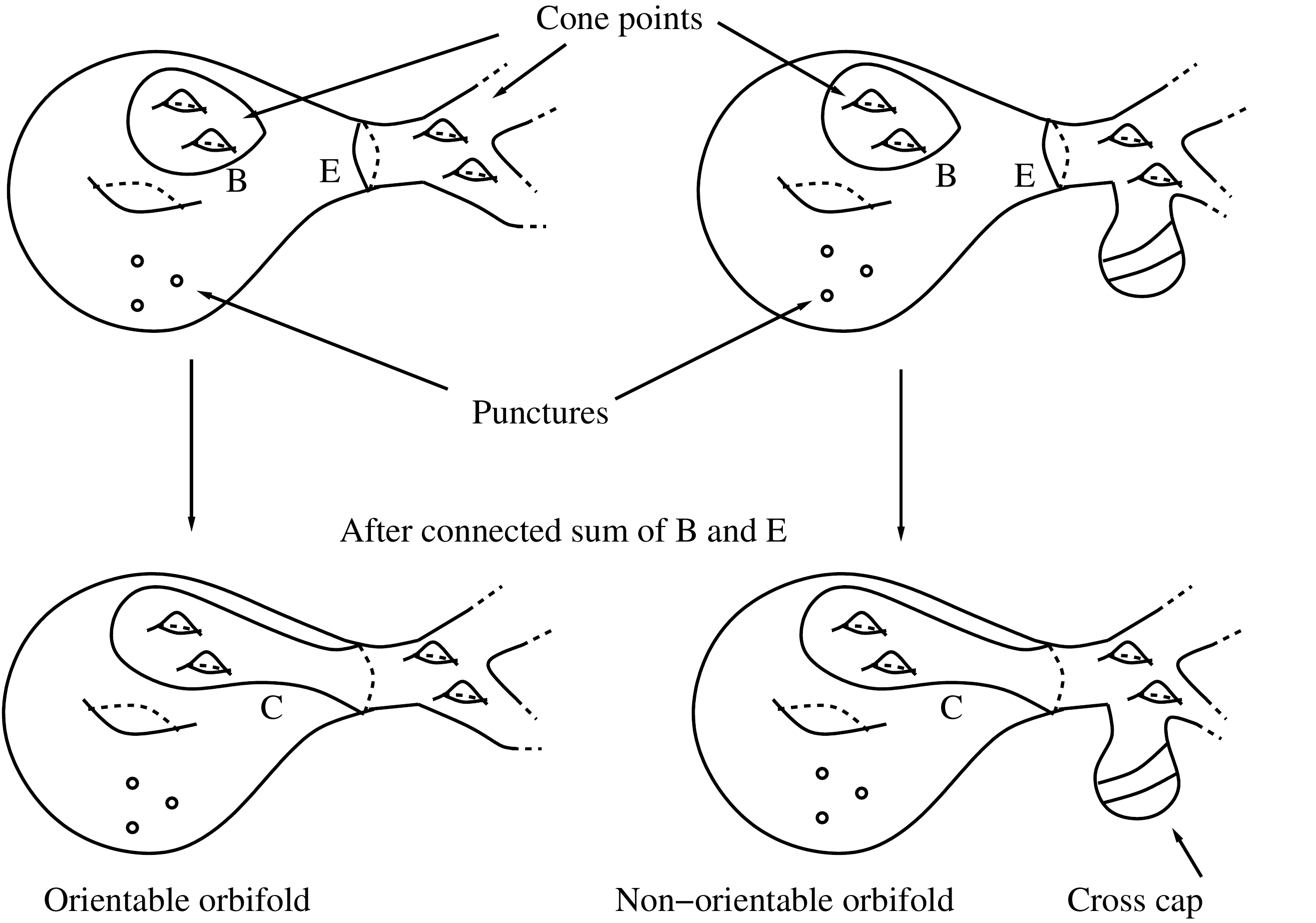}}

\centerline{Figure 3: Orientable and non-orientable orbifolds.}

\medskip

  In the first possibility, clearly, $\pi_1^{orb}(M)$ has trivial center, since
  it is isomorphic to the free product of more than one finite cyclic
  groups (see the proof of lemma \ref{lemmasub}). And
  in the second case, it is clear that $\pi_1^{orb}(M)$ has an amalgamated
  free product structure over an infinite cyclic group, with one of
  the factors non-abelian free.
  Hence, it has trivial center, by Lemma \ref{amal}.

  In the last two cases, 
  $\pi_1^{orb}(M)$ also has an amalgamated free product structure over
  an infinite cyclic group, with one of the factors non-abelian free, and hence has
  trivial center by Lemma \ref{amal}. To prove this, we will split $M$ 
  along a separating $\pi_1^{orb}$-injective embedded circle such
  that one piece contains all the cone points, and 
  the other one is smooth and has non-abelian free fundamental group. This can be done
  by choosing a nice orientable sub-orbifold of genus one with non-empty 
  boundary and which 
  contains no cone point.

  We describe this choice of a nice sub-orbifold in Figure 3.

  Note
  that, in the orientable case it has a torus as a connected sum
  factor, since it has genus $\geq 2$. Next, choose a disc on this torus
  which contains all the cone points lying on the left of the neck of
  the torus as in Figure 3,
  and then take the connected sum $C$ of the boundary $B$ of this disc and a 
  separating closed curve $E$ around the neck. Then, $C$ is the desired
  separating circle, by applying Lemma \ref{lemmasub}.

In the non-orientable case it has genus $\geq 3$, and hence there are
at least three cross caps. We can replace three of these cross caps by
an orientable genus one surface with a single cross cap (see
\cite{D88}).
Then, similar to the orientable case, 
we can choose a separating circle $C$ as shown in Figure 3.

  Then, again we apply the Van-Kampen theorem and Lemma \ref{amal} to see
  that the center of $\pi_1^{orb}(M)$ is trivial.
\end{proof}

An immediate corollary of Proposition \ref{center-trivial} is the following.

\begin{cor}\label{center-cor} Let $M\in {\cal C}_0\cup {\cal C}_1$
  be as in Proposition \ref{center-trivial}, and $\wt M$ is
  obtained from $M$ after removing a finite number of regular
  points. Then, $\wt M$ also satisfies the hypothesis of Proposition
  \ref{center-trivial}, and hence the center of   $\pi_1^{orb}(\wt M)$ is 
  trivial.\end{cor}

\section{Proofs of the theorems}
We now prove the two theorems of Section 2.

\begin{proof}[Proof of Theorem \ref{thm1}] Recall that $N$ is a nice
  sub-orbifold of $M\in {\cal C}_0\cup {\cal C}_1$. At first we prove the injectivity
  of (\ref{PB}), that is of ${\cal {PB}}_n(N)\to {\cal {PB}}_m(M)$, when $n=m$. The proof is by induction
  on $n$. Recall that $PB_1(X)=X$ for any orbifold $X$.
  Hence, by Proposition \ref{sub}, ${\cal {PB}}_1(N)\to {\cal {PB}}_1(M)$ is injective.
  Therefore, assume that (\ref{PB}) is true for $n=k-1$ and for any nice sub-orbifold of an
  orbifold from ${\cal C}_0\cup {\cal C}_1$.

  Next, consider the exact sequence
  (\ref{1.3}) for $r=1$. Then, we have the following commutative diagram of exact sequences.

\begin{align}\begin{gathered}\label{PBProof}
  \xymatrix@-.7pc{1\ar[r]&{\cal {PB}}_{k-1}(\wt N)\ar[r]\ar[d]&{\cal
      {PB}}_k(N)\ar[r]\ar[d]&{\cal {PB}}_1(N)\ar[r]\ar[d]&1\\
  1\ar[r]&{\cal {PB}}_{k-1}(\wt M)\ar[r]&{\cal {PB}}_k(M)\ar[r]&{\cal {PB}}_1(M)\ar[r]&1.}\end{gathered}\end{align}

Note that, $\wt N$ is a nice sub-orbifold of $\wt M$, since $\wt N$ is obtained from $N$
after removing one regular point and $\wt M$ is obtained from $M$ after removing the same
regular point (see Remark \ref{nice-remark}).

Therefore, in the commutative diagram (\ref{PBProof}), the first vertical homomorphism is injective
by the induction hypothesis, and the last one is injective from Proposition \ref{sub}. Hence, by a simple
diagram chase it follows that the middle homomorphism is injective as well. This proves
the injectivity of (\ref{PB}) when
$n=m$.

Now, we come to the proof of the injectivity of (\ref{PB}) when $n\leq m$. Consider the
following diagram. Here, $p:PB_m(M)\to PB_n(M)$ is the projection map to the first $n$ coordinates.

\begin{align}\begin{gathered}\label{PBProofG}
    \xymatrix{{\cal {PB}}_n(N)\ar[r]\ar[dr]&{\cal {PB}}_m(M)\ar[d]^{p_*}\\
      &{\cal {PB}}_n(M).}\end{gathered}\end{align}

The slanted homomorphism is injective by the previous case, and hence the top homomorphism is also injective. This
proves the injectivity of (\ref{PB}).

Next, we come to the proof of the injectivity of (\ref{B}), that is of ${\cal B}_n(N)\to {\cal B}_m(M)$. 
Consider the following commutative diagram of the exact sequences (\ref{1.1}) for
$N$ and $M$. Here, $i:S_n\to S_m$ is the inclusion map.

\begin{align}\begin{gathered}\label{BProof}
  \xymatrix{1\ar[r]&{\cal {PB}}_n(N)\ar[r]\ar[d]&{\cal B}_n(N)\ar[r]\ar[d]&S_n\ar[r]\ar[d]^{i}&1\\
  1\ar[r]&{\cal {PB}}_m(M)\ar[r]&{\cal B}_m(M)\ar[r]&S_m\ar[r]&1.}\end{gathered}\end{align}

Once again a simple diagram chase and using the injectivity of (\ref{PB}) we conclude that the
middle vertical homomorphism (that is, (\ref{B})) in the above diagram  is injective.

This completes the proof of Theorem \ref{thm1}.
\end{proof}

Now, we come to the proof of the triviality of the center of the (pure) orbifold braid groups.

\begin{proof}[Proof of Theorem \ref{center}]
  At first we prove the theorem for ${\cal {PB}}_n(M)$, that is, we
  show that ${\cal {PB}}_n(M)$ has trivial center. The proof is by induction
  on $n$. By Proposition \ref{center-trivial}, for $n=1$, ${\cal {PB}}_1(M)=\pi_1^{orb}(M)$ has 
  trivial center. Therefore, we assume that ${\cal {PB}}_{n-1}(M)$ has trivial center.
  
Now, consider the exact sequence (\ref{1.3}) for $r=n-1$.
  
  \begin{align}\label{center-exact}
  \xymatrix@-.5pc{1\ar[r]&{\cal {PB}}_1(\wt M)\ar[r]&{\cal
                                                      {PB}}_n(M)\ar[r]&{\cal
                                                                        {PB}}_{n-1}(M)\ar[r]&1.}\end{align}

Since, by hypothesis, ${\cal {PB}}_{n-1}(M)$ has trivial center and
 by Corollary \ref{center-cor}, ${\cal {PB}}_1(\wt M)$ has trivial
 center, we can apply Lemma \ref{extension} to conclude
 that  ${\cal {PB}}_n(M)$ has trivial center.

Next, we prove that the center of ${\cal B}_n(M)$ is trivial for
$n\geq 3$. Recall the exact sequence (\ref{1.1}).

\begin{align}\label{Bcenter}
  \xymatrix{1\ar[r]&{\cal {PB}}_n(M)\ar[r]&{\cal B}_n(M)\ar[r]&S_n\ar[r]&1.}\end{align}

Note that, $S_n$ has trivial center for $n\geq 3$, and from the
previous case ${\cal {PB}}_n(M)$ has trivial center. Hence, by Lemma
\ref{extension} ${\cal {B}}_n(M)$ has trivial center.

Finally, for
$n=1$, $\pi_1^{orb}(M)={\cal {PB}}_1(M)={\cal {B}}_1(M)$ and hence
by Proposition \ref{center-trivial}, its center is trivial.

This completes the proof of Theorem \ref{center}.
\end{proof}

We conclude with the following remark.

\begin{rem}\label{last}{\rm For an action of a discrete group $G$    
    on a connected $2$-manifold $M$, one considers the {\it orbit
    configuration space} ${\cal O}_n(M,G)$ of $n$ points.
    By definition, ${\cal O}_n(M,G)$ is the space of all
    $n$-tuples of points of $M$ with pairwise
    distinct orbits. The orbit configuration space is an immediate
    generalization of the configuration space, since 
    $PB_n(M)={\cal O}_n(M, \l 1\r )$. During the last couple of
    decades orbit configuration spaces were of much interest and
    many works were done assuming the action is free and properly
    discontinuous, since the Fadell-Neuwirth fibration theorem still holds
    in this case. 

    We now assume that the action is properly discontinuous and effective
    (need not be free), 
    with isolated fixed points. Then, we note here that the Fadell-Neuwirth
    fibration theorem does not hold in this 
    generality of non-free actions (see Lemma 2.3 in \cite{Rou22}).
    Nevertheless, using (\ref{1.3}), the following exact sequence was proved in \cite{Rou22}, assuming that 
    $M\neq {\Bbb S}^2, {\Bbb {RP}}^2$, and that when $M/G$ has genus zero, then
    it either has a puncture or has nonempty boundary.  See
    \cite{Rou21} and \cite{Rou22} for some more on this matter.
    
    \begin{align}
  \xymatrix@-.5pc{1\ar[r]&\pi_1({\cal O}_{n-r}(M_r, G))\ar[r]&
    \pi_1({\cal O}_n(M, G))\ar[r]&
    \pi_1({\cal O}_r(M, G))\ar[r]&1.}\end{align}
\noindent
Here, $M_r$ is the complement in $M$ of the orbits of $r$ points, which
are not fixed points of the action.
    From this exact sequence we can also prove the triviality of the center of
    $\pi_1({\cal O}_n(M,G))$ for most such pairs $(M,G)$, following the
    proof of Theorem \ref{center}.}\end{rem}

    \newpage
\bibliographystyle{plain}
\ifx\undefined\bysame
\newcommand{\bysame}{\leavevmode\hbox to3em{\hrulefill},}
\fi

\end{document}